\documentclass[12pt,a4paper]{amsart}

\usepackage{amsmath,amsfonts,amssymb,amsthm,mathrsfs}

\usepackage{graphicx}
\usepackage{color}

\allowdisplaybreaks[1]

\usepackage{marginnote}

\newcommand{\be}{\begin{equation}}
\newcommand{\ee}{\end{equation}}
\newcommand{\bea}{\begin{eqnarray}}
\newcommand{\eea}{\end{eqnarray}}
\newcommand{\bean}{\begin{eqnarray*}}
\newcommand{\eean}{\end{eqnarray*}}
\newcommand{\rf}[1]{(\ref {#1})}

\newcommand{\xn}{|\!|\!|}

\newcommand{\lM}{\underline M}
\newcommand{\uM}{\overline M}

\def\dx{\,{\rm d}x}
\def\dy{\,{\rm d}y}

\def\dr{\,{\rm d}r}

\def\ds{\,{\rm d}s}

\def\drho{\,{\rm d}\rho}

\def\eps{\varepsilon}
\newcommand{\ep}{\varepsilon}

\def\s{\sigma}

\def\r{\varrho}
\def\xn{|\!|\!|}

\def\mn{|\!\!|}

\def\mn2{|\!\!|_{M^{d/2}}}

\def\y{y_\ast}

\newcommand{\M}{\mathcal {M}}

\newcommand{\re}{\mathbb R}
\newcommand{\R}{\mathbb R}

\newcommand{\N}{\nabla}

\renewcommand{\ep}{\eps}
\newcommand{\lam}{\lambda}

\newcommand{\Del}{\Delta}

\newcommand{\sg}{\sigma}

\renewcommand{\r}{\rho}

\def\<{\langle }

\renewcommand{\qed}{\qquad\kern1pt   
   \vbox{\hrule height 0.6pt      
         \hbox{\vrule width 0.6pt 
               \vbox{\vskip 6pt}  
               \hskip 3pt
              \vrule width 1.3pt} 
         \hrule depth 1.3pt}     
   \kern1pt}

\newcommand{\eq}[1]{\begin{equation}#1\end{equation}}%
\newcommand{\spl}[1]{{\begin{split}#1\end{split}}}
\newcommand{\eqn}[1]{\begin{equation*}#1\end{equation*}}

\newcommand{\rd}{\color{red}}

\renewcommand{\rd}{\color{black}}

\newtheorem{theorem}{Theorem}
\newtheorem{proposition}[theorem]{Proposition}
\newtheorem{lemma}[theorem]{Lemma}
\newtheorem{corollary}[theorem]{Corollary}

\theoremstyle{definition}

\def\qed{\hfill $\square$}

\theoremstyle{remark}
\newtheorem{remark}[theorem]{Remark}

\numberwithin{equation}{section}
\numberwithin{theorem}{section}

\author[P. Biler]{Piotr Biler}
\address[P. Biler]{
 Instytut Matematyczny, Uniwersytet Wroc\l awski,
 pl. Grunwaldzki 2/4, \hbox{50-384} Wroc\-\l aw, Poland}
\email{Piotr.Biler@math.uni.wroc.pl; \ orcid.org/0000-0002-2293-6964}

\author[G. Karch]{Grzegorz Karch}
\address[G. Karch]{ 
 Instytut Matematyczny, Uniwersytet Wroc\l awski,
 pl. Grunwaldzki 2/4, \hbox{50-384} Wroc\-\l aw, Poland}
\email{Grzegorz.Karch@uwr.edu.pl;\ orcid.org/0000-0001-9390-5578}

\author[H. Wakui]{Hiroshi Wakui}
\address[H. Wakui]{ 
 Instytut Matematyczny, Uniwersytet Wroc\l awski,
 pl. Grunwaldzki 2/4, \hbox{50-384} Wroc\-\l aw, Poland;\   Faculty of Science Division I, Tokyo University of Science, 1-3 Kagurazaka,
Shinjuku-ku, Tokyo 162-8601, Japan} 
\email{
hiroshi.wakui@rs.tus.ac.jp; \ orcid.org/0000-0002-4676-4669 }

\title[Parabolic-elliptic Keller--Segel model]{
Large self-similar solutions \\ of the parabolic-elliptic Keller--Segel model}

\begin{document}

\begin{abstract} 
We construct radial self-similar solutions of the, so called, minimal parabolic-elliptic Keller--Segel model in several space dimensions with radial, nonnegative initial conditions which are below the Chandrasekhar solution --- the singular stationary solution of this system. 
\end{abstract}



\date{\today}
\maketitle

\baselineskip=18.5pt

\section{Statement of the problem}

We begin our discussion of the following minimal Keller--Segel chemotaxis system
\begin{equation} \label{eq;DD}
 \spl{
  & u_t - \Del u + \N \cdot (u \N\psi) = 0,&&\quad t > 0,\ \  x \in \re^{d},\\
  &-\Del \psi = u,&&\quad t>0,\ \ x \in \re^{d},
 }
\end{equation}
with $d\ge 3$,  with noticing that 
it is preserved by the scaling transformation
\begin{equation}\label{scale-u}
  u_\lambda (t,x) =\lambda^{{2}}u(\lambda^{2}t, \lambda x), \quad 
  \psi_\lambda (t,x) =\psi (\lambda^{2}t, \lambda x)
  \ \ \ {\rm for\ every}\ \ \lambda>0, 
\end{equation} 
namely, if 
$(u,\psi)$ is a solution to system  \eqref{eq;DD} then so is 
$(u_\lambda,\psi_\lambda)$. 
Each solution invariant under this scaling, i.e.   satisfying 
\begin{equation}\label{u:ul}
  u(t,x) =u_\lambda(t,x), \quad \psi(t,x)=\psi_\lambda(t,x)\quad \text{for all}\quad t > 0,\;  x\in \mathbb{R}^{d},\; 
\text{and} \; \lambda > 0,
\end{equation}
 is called  a {\it self-similar solution} to system \eqref{eq;DD}. 
Choosing $\lambda^2=\frac{1}{t}$ in  equations \eqref{u:ul},
 we obtain that each self-similar solution has the form
\begin{equation}\label{eq;self-similar}
 u(t,x)=\frac{1}{t}U\left(\frac{x}{\sqrt{t}}\right),
 \quad
  \psi (t,x)=\Psi\left(\frac{x}{\sqrt{t}}\right)  
\end{equation} 
with 
\begin{equation}\label{eq;self-similar2}
U(x)=u(1,x)
 \quad
 \text{and} \quad \Psi(x)=\psi(1,x).
\end{equation} 
If a self-similar solution \eqref{eq;self-similar}  corresponds to an initial datum, namely, if the limit 
$$
u_0(x)\equiv \lim_{t\to 0} \frac{1}{t}U\left(\frac{x}{\sqrt{t}}\right)
$$ 
exists  (for example, in the sense of distributions),
then the initial datum $u_0$ has to be homogeneous of degree $-2$.
In the two dimensional case, nonnegative self-similar solutions of system \eqref{eq;DD}
correspond to  multiples of the Dirac measure supported in the origin and the existence of such solutions is well-known -- see references to this result in the next section. 
In our main result stated in the following theorem, we consider the case $d\geq 3$, and we construct radial, nonnegative  self-similar solutions to system \eqref{eq;DD} corresponding to  initial data of the form  
$u_0(x)=\frac{C}{|x|^2}$ for some constant $C>0$.

\begin{theorem}\label{main}
Let $d\geq 3$. 
System \rf{eq;DD} supplemented with the initial condition 
\begin{equation}\label{u0:m}
 u_{0}(x) 
 =
 \varepsilon 
 \frac{ 2(d-2)}{|x|^{2}} 
 \qquad 
 \text{with} 
 \quad 
 0<\varepsilon < 1,
\end{equation}
has a self-similar solution of the form \eqref{eq;self-similar},
where the self-similar profile $U$
 is a nonnegative and radial function  satisfying 
 $U\in C^\infty(\re^d)\cap L^\infty (\re^d)$ 
as well as the estimate
\begin{equation}\label{U:R}
\sup_{R>0} R^{2-d}\int_{\{|x|<R\}} U(x)\dx\leq 2 \s_d \varepsilon 
\quad
 {\rm and}
\quad
\lim_{R \to \infty}
R^{2-d}
\int_{\{ |x| < R \}}
U(x)
\, {\rm d}x
=
2\sg_{d} \ep
.
\end{equation}
Moreover, we have  $\nabla\Psi =\nabla E_d*U$ with 
$E_d(x)=\frac{1}{(d-2)\s_d}|x|^{2-d},$ where the number $\s_d=\frac{2\pi^{{d}/{2}}}{\Gamma\left(\frac{d}{2}\right)}$ denotes the measure of the unit sphere $\mathbb{S}^{d-1}\subset \mathbb{R}^d$.
\end{theorem}

In this work, we limit ourselves to radial self-solutions, although 
there are nonradial self-similar solutions to system   \rf{eq;DD} with $d\geq 3$ corresponding to sufficiently small initial data which are homogeneous of degree $-2$. 
We recall such results below, in the next section.

 The upper bound for $\varepsilon$ in the initial condition \rf{u0:m} is related   to the assumption that $u_0(x)$ stays  below  the  singular stationary solution (the {\em Chandrasekhar solution}) system \rf{eq;DD}    explicitly given by 
\eq{\label{eq;singular-stationary-sol}
 u_{C}(x)
 =
 \frac{2(d-2)}{|x|^{2}} \qquad  
(\text{with} \quad  \nabla\psi_C =\nabla E_d*u_C).
 }%
 This explicit solution  plays an important role in our analysis and, for example, 
 is used to obtain (with a suitable comparison principle)  estimate \rf{U:R} because 
$$ 
 R^{2-d}\int_{\{|x|<R\}} u_C(x)\dx 
 =
 2(d-2)\sg_{d} R^{2-d}\int_{0}^{R} s^{d-3} \ds
 =2\sg_{d}
 \quad \text{for each}\quad R>0.
$$ 

Self-similar solutions constructed in Theorem \ref{main} describe the large time behavior of a large class of other solutions to the Cauchy problem for system \eqref{eq;DD} and we discuss it below in Corollary \ref{cor;self-similar-asymptotics}.

We postpone further comments on Theorem \ref{main} to Section \ref{sec:Thm1}, and now we recall  results  on the nonexistence of self-similar solutions.
Note first that the presence of the singular self-similar solution \eqref{eq;singular-stationary-sol} does not contradict existence of other (not necessarily smooth) self-similar solutions corresponding to the initial condition \eqref{u0:m} with either $\varepsilon= 1$ or $\varepsilon>1$
(see the next section for related comments concerning  the semilinear   heat equation). 
The following remark  states, however, that this number cannot be too large.
 
 \begin{remark} 
\label{main:blowup}
 System \rf{eq;DD} supplemented with the initial condition \rf{u0:m} cannot have any local-in-time solution (hence, self-similar solutions neither) if 
\eq{\label{C(d):0}
 \varepsilon > 
 C(d)
 \equiv 
 \frac{16}{\Gamma\left(\frac{d}{2}\right)}
 \int_0^\infty  
  {\rm e}^{-\r^2}\r^{d+1}(2(d-2)+4\r^2)^{-1}{\rm d}
 \r,}%
which follows from \cite[Theorem 2.2]{BZ-JEE}. 
Below, in Proposition \ref{prop;bounds-C(d)}  
we show that the integral defining the constant $C(d)$ in relation \rf{C(d):0} 
(which can be expressed in terms of the incomplete Gamma function)
satisfies the estimates
\eq{\label{bound-C(d)}
 1
 <
 \frac{2}{d-1}
 \left(
  \frac{\Gamma\left(\frac{d+1}{2}\right)}{\Gamma\left(\frac{d}{2}\right)}
 \right)^2 
 < 
 C(d) 
 <
 \left(
  \frac{2}{d-2}
 \right)^{\frac12}
 \frac{\Gamma\left(\frac{d+1}{2}\right)}{\Gamma\left(\frac{d}{2}\right)} 
 <
 \frac{d-1}{d-2}
 \le 
 2.
}%
In particular, the integral in expression \eqref{C(d):0} satisfies    $C(d)\to 1$ as $d\to\infty$, thus asymptotically for large dimensions, both Theorem  \ref{main} and Remark \ref{main:blowup} provide  ``almost''  optimal assumptions on either the existence or nonexistence of radial, nonnegative self-similar solutions. 
 \end{remark}

 The remainder of this paper is constructed in the following way.
In the next section, we recall our motivations to study the Keller--Segel model \rf{eq;DD} and its self-similar solutions. 
 In Section \ref{sec:Thm1}, we present  main ideas of the proof of Theorem \ref{main} and we comment the obtained result on the existence of self-similar solutions. 
 A comparison principle for radial distributions of sufficiently regular solutions to system  
  \rf{eq;DD} is proved in Section~\ref{sec:comp}.
Self-similar solutions from Theorem \ref{main} are constructed in Section \ref{sec:self} by a suitable approximation procedure.
Regularity of self-similar profiles $U$ is shown in Section \ref{sec:reg}. 
The proof of estimates in Remark \ref{main:blowup} is in Section \ref{sec:non}.

\section{Review of other results on self-similar solutions }\label{sec:known-result-self-similar-solutions}

Our motivations to study system \eqref{eq;DD}
 come from Mathematical Biology, where these equations  are a simplified Keller--Segel  system modeling chemotaxis, see e.g. \cite{B-AMSA,B-book,BDP}. 
The unknown variables $u=u(t,x)$ and $\psi=\psi(t,x)$ denote the density of the population of microorganisms (e.g. swimming bacteria or slime mold),  and the density of a chemical secreted by themselves that attracts them and makes them to aggregate, respectively. 
Another important interpretation of system \eqref{eq;DD} comes from Astrophysics, where the unknown function $u=u(t,x)$ is the density of gravitationally interacting massive particles in a cloud (of stars, nebulae, etc.), and $\psi=\psi(t,x)$ is the Newtonian potential (``mean field'') of the mass distribution $u$, see \cite{B-SM,BHN,CSR}.  
System \rf{eq;DD} can also be interpreted as a drift-diffusion equation 
(see equation~\rf{KS} below)  with a linear diffusion and a quadratic nonlocal transport terms.

Self-similar solutions to such an equation play an important role in a study of large time asymptotics of other solutions to the Cauchy problem and have been already studied by several  authors.
Let us recall some results from  previous works where self-similar solutions have been constructed either to  the parabolic--elliptic system \eqref{eq;DD} or to the doubly parabolic Keller--Segel system when the equation for $\psi$ is replaced by a linear diffusion equation $\tau \psi_t=\Delta \psi+u$ with a fixed parameter $\tau>0$.

It is well-known that, in the  case $d=2$, a self-similar solution to  the parabolic-elliptic system  \eqref{eq;DD}
exists for each initial datum $u(\cdot,0)=M\delta_0$ with the Dirac measure $\delta_0$ and for each $M\in(0,8\pi)$.
Such solutions  are unique and smooth for $t>0$, moreover, 
other    global-in-time solutions to the Cauchy problem for 
 \eqref{eq;DD} with  initial data satisfying 
with $M\equiv \int_{\R^2} u_0(x)\dx<8\pi$ have an asymptotically self-similar large time behavior. We refer to works \cite{NSY02, NS04, NS08, BDP,BKLN2,CM17,BM14,CEM14}    and to references therein for proofs and for a discussion of such  results.   

On the other hand,  for the doubly parabolic Keller--Segel system the existence of ``large'' self-similar solutions depend in a sensitive way on the value of the coefficient $\tau>0$, see \cite{Mizutani-Muramoto-Yoshida,Mizutani-Nagai,BCD}. 
In particular, it is shown in \cite{BCD} that there is a  unique self-similar radial solution for each $M\in[0,8\pi)$ but for $\tau\gg 1$ there exist also nonunique self-similar radial solutions with each $M\in(8\pi,M(\tau))$, and $M(\tau)\to\infty$ as $\tau\to\infty$.

For $d\geq 3$, self-similar solutions of the parabolic-elliptic  problem \rf{eq;DD} 
as well as of its doubly parabolic counterpart 
have been constructed for small initial data in various function spaces with norms invariant under  scaling properties \eqref{scale-u}.
A construction of those self-similar solutions heavily depends on the semigroup approach which usually need  smallness assumption imposed  on initial data. 
Several results on 
the existence of  self-similar solutions with small initial conditions in scaling invariant spaces can be found 
e.g. in \cite{B-SM,B-AMSA,BCGK,K-JMAA,Kozono-Sugiyama,Lem,NS08}.
We refer also to the work by Senba  \cite[Theorem 3]{Senba-05} containing  a result on ``big'' self-similar solutions of the parabolic-elliptic model \eqref{eq;DD}.

Results of this work have been partially motivated by results on self-similar solutions 
of the Cauchy problem for the nonlinear heat equation
\begin{align}
&u_t=\Delta u+|u|^{\alpha}u,\qquad t > 0,\ \  x \in \re^{d},\label{nheat1}\\
&u(0,x)=\varepsilon  u_C(x),\label{nheat2}
\end{align}
where the function $u_C(x)=A|x|^{-{2}/{\alpha}}$, with a certain explicit number  $A=A(d,\alpha)>0$, is a  stationary singular solution to the nonlinear heat equation. For $\alpha=1$, this is  a counterpart of the Chandrasekhar solution  \eqref{eq;singular-stationary-sol}
for the Keller--Segel system. 
For $d\geq 3$, $\alpha >\frac{2}{d-2}$, and $\varepsilon \in (0,1)$, problem   \eqref{nheat1}--\eqref{nheat2} has a smooth self-similar solution, see \cite[Lemma 10.3]{GV97}.
Results on the existence of self-similar solutions to this problem with  $\varepsilon=1$ are contained in \cite[Lemma 10.1 and Theorem 10.4]{GV97}.
Moreover, by results in \cite{SW}, 
if $\frac{2}{d-2}<\alpha<\frac{4}{d-4-2\sqrt{d-1}}$, then there exists $\varepsilon>1$ (close to $1$) such that 
problem \eqref{nheat1}--\eqref{nheat2} has 
 two smooth self-similar solutions.


\section{Comments on Theorem \ref{main} and ideas of its proof}\label{sec:Thm1}

Since the second equation in system   \rf{eq;DD}  is not uniquely solved with respect to $\psi$,  we always assume that 
\be
\nabla \psi=\nabla E_d\ast u  \label{fundsol}
\qquad \text{with} \quad 
 E_d(x)=\tfrac{1}{(d-2)\s_d}\tfrac{1}{|x|^{d-2}}\quad \text{for}\quad d\ge 3. 
\ee 
Consequently, we consider in fact the Cauchy problem for the nonlocal transport equation 
\begin{equation}\label{KS}
u_t-\Delta u+\nabla\cdot(u\nabla E_d\ast u)= 0,  \qquad t>0, \ \ x\in\R^d.
\end{equation}
For radially symmetric solutions,  
we transform  system \eqref{eq;DD} (or equation \eqref{KS}) supplemented with an initial condition $u(0,x)=u_0(x)$
into the problem for the radial mass distribution function
(see e.g. \cite[equation (12)]{BHN} and calculations in the proof of Theorem \ref{main})  
\be
M(t,r)=\int_{\{|x|<r\}} 
u (t,x) \dx =\s_d\int_0^r u(t,s)s^{d-1}\,{\rm d}s \quad \text{with}\quad  r=|x|,
\label{M-u}
\ee 
namely, the initial-boundary value problem 
\eq{\label{eq;local-mass}
 \spl{ 
 &M_{t} - M_{rr}   + \frac{d-1}r M_{r} - \frac{1}{\sg_{d}r^{d-1}} MM_{r}=0,&&\quad t>0, \ \ 0<r<\infty,\\
 & M(t,0) =0,&&\quad t>0,\\
 & M(0,r) =M_{0}(r),&&\quad 0<r<\infty.
 }
}%
Of course, from relation \rf{M-u}  then we have 
\be 
u(t,x)=\frac{1}{\s_dr^{d-1}}M_r(t,r) \qquad \text{with} \quad  r=|x|.\label{u-M}
\ee

Since   system \eqref{eq;DD} has the scaling invariance \rf{scale-u}, the first  equation in \eqref{eq;local-mass} is also invariant under the scaling 
\be\label{scale-M}
 M_{\lam}(t,r)
 =\lam^{2-d}M(\lam^{2}t,\lam r)\qquad \text{for each \ \ $\lam > 0$. }
\ee
Thus, in order to construct self-similar solutions to  
  problem \eqref{eq;local-mass}, we choose  
   the following initial datum (cf.  equation \eqref{u0:m}) 
 \begin{equation}\label{M0:s}
 M_{0}(r)= \int_{\{|x|<r\}}\varepsilon \frac{2(d-2)}{|x|^{2}}\dx=
    \varepsilon 2\s_d r^{d-2}.
\end{equation}

Our approach, based on this alternative formulation of system \rf{eq;DD} for radially symmetric functions, allows us to deal with more regular functions $M(t,r)$ satisfying  an evolution PDE in one space variable, however, with singular coefficients. 
Here, we are carrying our analysis  of solutions to problem \eqref{eq;local-mass}
 using some ideas from  \cite{YY,BKZ-AiM,BKP} and, in particular, 
 a rather subtle new comparison principle in  Theorem \ref{comp} below. 
This comparison principle implies that the whole sequence of suitable approximating solutions of problem \eqref{eq;local-mass}
converges to a solution of the same problem with  a singular initial datum \eqref{M0:s}, not merely a subsequence. 
At the same time, this leads to the uniqueness of solutions 
to problem  \eqref{eq;local-mass} with the homogeneous initial datum \eqref{M0:s}
 as well as to the scaling invariance of solutions. In that way, self-similar solutions 
 to problem \eqref{eq;DD} and \eqref{u0:m}
 are constructed for $\ep<1$.


\section{Comparison principle for radial mass distribution function} \label{sec:comp}

We begin by recalling the recent  result from \cite{BKP} on the existence of global-in-time solutions to the Cauchy problem for system \rf{eq;DD} corresponding to sufficiently regular (in the scale of Morrey spaces) but not necessarily small initial data.  
Recall here  that the homogeneous Morrey spaces  $M^s(\R^d)$ are defined  by their norms 
\bea
|\!\!| u|\!\!|_{M^s}&\equiv& \sup_{R>0,\, x\in\R^d}R^{d(1/s-1)} 
\int_{\{|y-x|<R\}}|u(y)|\dy.\label{hMor}
\eea  
Moreover, 
the  {\em radial concentration} of a locally integrable nonnegative function $u$ is defined by  
\be
\xn u\xn\equiv\sup_{R>0}R^{2-d}\int_{\{|y|<R\}} u(y) \dy.\label{r-conc}
\ee
Clearly, $\xn u\xn\le |\!\!|u|\!\!|_{M^{{d}/{2}}}$ but, in fact, these are equivalent quantities for nonnegative radial functions, see e.g. \cite[Lemma 7.1]{BKZ-AiM}. 
The role of the notion of the radial concentration and the Morrey space norms in the analysis of the Keller--Segel model  is explained and discussed in \cite{BKZ-AiM,BKP}.

\begin{theorem}[{\cite[Theorem 2.1]{BKP}}]
\label{glo}
Let $d\geq 3$. 
Assume that a radially symmetric nonnegative initial condition $u_0\in M^{d/2}(\R^d)$ satisfies  
\be\label{a1}
\xn u_0\xn\equiv\sup_{R>0}R^{2-d}\int_{\{|x|<R\}} u_0(x)\dx<2\s_d.
\ee 
There exists $p\in \left(\frac{d}{2},d\right)$ such that if, moreover, $u_0\in M^p(\R^d)$, then the corresponding global-in-time solution $u=u(t,x)$ of problem \rf{eq;DD} exists in the space
\begin{equation}\label{3*}
  {\mathcal C}_w \Big( [0,T), M^{d/2}(\R^d) \cap M^p(\R^d) \Big)\cap \Big\{ u: (0,T) \to L^\infty (\R^d): { \sup_{0 < t < T}} t^{\frac{d}{2p}}\| u(t)\|_\infty <\infty \Big\}
\end{equation}
for each $T>0$. Moreover,  this solution  is smooth, nonnegative,  radial, and satisfies the bound 
\begin{equation}\label{global_sol}
\xn u(t)\xn=\sup_{R>0}R^{2-d}\int_{\{|x|<R\}} u(t,x)\dx< 2\s_d\qquad  
\text{for all}\quad  t>0.
\end{equation}
\end{theorem}

Let us formulate certain  properties of solutions constructed in Theorem \ref{glo} in a form suitable for this work. They are 
 either a direct consequence of Theorem \ref{glo}
 or  proved in the paper \cite{BKP}.

\begin{proposition}\label{prop;?}
 Under assumptions of Theorem \ref{glo}, let $u$ be the radial global-in-time solution  corresponding to the initial datum $u_{0} \in M^{d/2}(\re^{d}) \cap M^{p}(\re^{d})$.
 Denote 
 $$
  M(t,r)
  =\int_{\{|x|<r\}}
u(t,x) \dx \quad 
  \text{\rm and}
 \quad
  M_{0}(r)
  =
  \int_{\{|x|<r\}} 
u_{0}(x) \dx.
 $$
 Then $M$ has following properties:
 \begin{enumerate}
  \item $M \in {\mathcal C}^{1,2}( (0,\infty) \times (0,\infty) ) \cap {\mathcal C}( [0,\infty) \times [0,\infty) )$. 
  \item  $M_{r}(t,r) \ge 0$ for all $t>0$ and $r>0$.
  \item There exist $p\in \big(\frac{d}{2},d\big)$, 
  $\varepsilon  \in \big(0,\frac{d}{2p}\big)$, and $K>0$ such that
  $$
  0 
   <
   M(t,r)
   \le
   \min 
   \{ 
    \ep 2 \sg_{d}r^{d-2}, Kr^{d-d/p}
   \}
\qquad \text{for all \; $t >0$ and $r>0$.}
  $$
  \item
  For each $T>0$ and all $t\in (0,T]$ 
  there exist $c_0=c_0(T)>0$ and $c_1=c_1(T)>0$ such that 
  $$
    M(t,r)
    \le 
    c_{0}
    {t^{-\frac{d}{2p}}} 
    r^d 
   \quad
   \text{and}
   \quad 
    M_{r}(t,r) 
    \le 
    c_{1}t^{-\frac{d}{2p}} r^{d-1}
   \qquad \text{for all \ \ $r \ge 0$. }
  $$
 \end{enumerate}
\end{proposition}

\begin{proof}
The solution $u=u(t,x)$ constructed in Theorem \ref{glo} is smooth for $t>0$, hence, the regularity in item (1) results immediately from the definition of $M(t,r)$.
The function $u(t,x)$ is nonnegative thus the property (2) is a consequence of  formula \eqref{u-M}.
 The upper bound of $M(t,r)$ in item (3) plays a crucial role in the proof of Theorem \ref{glo}, and it has been proved in \cite[Proposition 4.1]{BKP}.
Since the solution $u(t,x)$ belongs to the space in \eqref{3*}, we have got 
${ \sup_{0 < t < T}} t^{\frac{d}{2p}}\| u(t)\|_\infty <\infty$ for each $T>0$ 
and consequently  the inequalities in item (4) follow from the definition of $M(t,r)$ and from equation \eqref{u-M}.
\end{proof}

The main goal of this section is to  extend the comparison principle from \cite[Proposition~4.1]{BKP} (a comparison of a subcritical solution with a special barrier function) to the case of general solutions constructed in Theorem \ref{glo}. 
Note that the applicability of comparison principles is usually restricted to sufficiently regular solutions, or more generally, sufficiently regular sub- and supersolutions of parabolic problems even with singular or degenerate coefficients  as those in problem  \rf{eq;local-mass}. 
In the context of equations for evolution problems for chemotaxis, related comparison principles are shown in \cite[Proposition 2.4]{YY} and \cite[Lemma 5.1]{BeWi}, the former in the case of the whole space and $u_0\in L^1(\R^d)$ densities (so bounded radial mass distribution functions 
$M$), the latter in finite domains $(0,T)\times(0,R)$ with $R<\infty$. 

\begin{theorem}[Comparison principle]\label{comp}
Consider functions $\lM$, $\uM\in{\mathcal C}^1([0,T]\times[0,\infty))\cap W^{2,\infty}_{\rm loc}((0,T)\times(0,\infty))$ such that 
\bea 
&&\lM(t,r)\le \eps 2\s_dr^{d-2},\ \ \ \uM(t,r)\le \eps 2\s_dr^{d-2},\ \ \eps\in(0,1),\label{subcrit}\\
&&\lM_r(t,r)\ge 0 \quad \text{and} \quad \uM_r(t,r)  \ge  0, \label{mono}\\
&&\text{either}\quad  
 \sup_{t>0, r>0}
  \frac{1}{\s_dr^{d-1}}\lM_r(t,r)\le b\ \quad \text{or}\quad  
  \sup_{t>0, r>0}
  \frac{1}{\s_dr^{d-1}} 
  \uM_r(t,r)\le b, \label{bdd}\\ 
&&\text{for some constant} \quad b>0,\nonumber\\
&&\lM_t\le \lM_{rr}-\frac{d-1}{r}\lM_r+\frac{1}{\s_dr^{d-1}}\lM\lM_r\quad \text{a.e. in} \quad (0,T)\times\R^d,\label{sub}\\
&&\uM_t\ge \uM_{rr}-\frac{d-1}{r}\uM_r+\frac{1}{\s_dr^{d-1}}\uM\uM_r\quad \text{a.e. in}\quad  (0,T)\times\R^d,\label{sup}\\
&&\lM(0,r)\le\uM(0,r),\label{ic}\\
&&\lM(t,0)\le\uM(t,0),\label{bc}\\
&& \sup_{t>0,\ r>0}r^{2-d}|\lM(t,r)-\uM(t,r)|<\infty.\label{diff}
\eea 
Then,  the following inequality holds true  
\be
\lM(t,r)\le \uM(t,r)\qquad 
\text{for every}\quad  (t,r)\in(0,T)\times(0,\infty).
\label{compM}
\ee 
\end{theorem}

\begin{proof}
The idea of the proof of the comparison principle is quite standard but we should be careful with the minimal regularity assumptions  on the functions to be compared. 
Note that relation \rf{diff} is a consequence of inequalities \rf{subcrit} but we prefer to keep it separately from size conditions \rf{subcrit}. 
For each   $\mu>0$ and $b>0$, $\nu>d-2\ge 1$  (which will be suitably chosen later on),
we consider the auxiliary function 
\be
z(t,r)=\lM(t,r)-\uM(t,r)-\mu{\rm e}^{2bt}( r+1)^\nu.
\label{zet} 
\ee 
We claim that 
$$
z(t,r)<0\ \ \ {\rm for\ every\ \ }(t,r)\in(0,T)\times(0,\infty).
$$
This will imply that 
$
 \lM(t,r)\le \uM(t,r)
$
for every 
$
 (t,r)\in(0,T)\times(0,\infty)
$. 
The function $z=z(t,r)$ is a continuous function and due to assumption \eqref{diff}
and because $\nu>d-2$, 
 there exists $(t_{0},r_{0}) \in [0,T] \times [0,\infty)$ such that
$
 z(t_{0},r_{0})
 =
 \max_{(t,r) \in [0,T] \times [0,\infty)} z(t,r)
$. 
Suppose {\em a~contrario} that 
$$
 z(t_0,r_0)=0 \quad \text{and}\quad z(t,r)<0 \quad \text{for all}\quad 
 t<t_0,\; r\geq 0. 
$$
In other words, the equality 
\be
 \lM(t_0,r_0)-\uM(t_0,r_0)
 =
 \mu{\rm e}^{2bt_0}(r_0+1)^\nu
 \label{values}
\ee
holds true. 
Note that, in fact, we have  $(t_0,r_0)\in(0,T]\times(0,\infty)$ by relations \rf{ic}--\rf{bc}, since $z=z(t,r)$ is strictly negative on the parabolic boundary of the domain $(0,T)\times(0,R)$ for sufficiently large $R>0$. 
Therefore, we obtain
\bea
z_t(t_0,r_0)\ge 0 \quad\text{and}\quad z_r(t_0,r_0)=0,\label{max1}
\eea
so, in particular 
\be
\lM_{r}(t_0,r_0)-\uM_{r}(t_0,r_0)=\mu{\rm e}^{2bt_0}\nu (r_0+1)^{\nu-1}. \label{gradient}
\ee
Using the $W^{2,\infty}_{\rm loc}(0,\infty)$ regularity, we obtain 
$$
z_r(t,r)=\int_{r_0}^rz_{rr}(t,\r)\drho,
$$
possibly except for a set $N\subset[r_0,\infty)$ of measure $0$, where inequalities \rf{sub}--\rf{sup} can also be violated. 
Then, there exists a sequence $(r_0,\infty)\setminus N\ni r_j\searrow r_0$ such that $z_{rr}(t_0,r_j)\le 0$.  
Indeed, otherwise $z_r(t_0,r_0)>0$ would hold, which contradicts equality in \rf{max1}. 
Thus, the  inequality 
\be
\lM_{rr}(t_0,r_0)-\uM_{rr}(t_0,r_0)
\le
\mu{\rm e}^{2bt_0}\nu(\nu-1) (r_0+1)^{\nu-2}\label{2der}
\ee
holds. 
Now, let us compute 
\begin{equation*}
\begin{split}
z_t=&\lM_t-\uM_t-2b\mu{\rm e}^{2bt}(r+1)^\nu\nonumber\\ 
\le&\lM_{rr}-\uM_{rr}-\frac{d-1}{r}\lM_r+\frac{d-1}{r}\uM_r +\frac{1}{\s_dr^{d-1}}\left(\lM\,\lM_r-\uM\,\uM_r\right)-2b\mu{\rm e}^{2bt}(r+1)^\nu\nonumber\\
 \le&\mu{\rm e}^{2bt_0}\nu(\nu-1) (r_0+1)^{\nu-2}+\frac{d-1}{r}\left(\uM_r-\lM_r\right) \nonumber\\
&+\frac{\lM_r}{\s_dr^{d-1}}(\lM-\uM) +\frac{\uM}{\s_dr^{d-1}}(\lM_r-\uM_r) -2b\mu{\rm e}^{2bt}(r+1)^\nu\nonumber
\end{split}
\end{equation*}
for all $r=r_j\searrow r_0$ and $t=t_0$.  
Observe that $\frac{\uM(t,r)}{\s_dr^{d-2}}\le 2\eps$, and $\sup_{(t,r)\in[t_0,T]\times[r_0,\infty)}\frac{\lM_r}{\s_dr^{d-1}}\le b<\infty$ by assumption \rf{bdd} (if we know that only  $\frac{\uM_r}{\s_dr^{d-1}}$ in assumption \rf{bdd} is uniformly bounded, then we write 
$
 \lM\,\lM_r-\uM\,\uM_r=\uM_r(\lM-\uM)+\lM(\lM_r-\uM_r)
$ and proceed analogously). 
 Passing to the limit and taking into account equations \rf{values},   \rf{gradient} and  inequality \rf{2der},  we obtain 
$$
z_t(t_0,r_0)\le \mu\nu (r_0+1)^{\nu-2}{\rm e}^{2bt_0}\left(\nu -1+(1+2\eps-d)\frac{r_0+1}{r_0}\right) -b\mu {\rm e}^{2bt_0}(r_0+1)^\nu.$$
Now, for $\eps<1$ we put $\nu=d-2\eps>d-2$, we obtain $z_t(t_0,r_0)<0$ which is  a contradiction with inequality in \rf{max1}. 
\end{proof}

\section{Existence of self-similar solutions}\label{sec:self}

A self-similar solution from Theorem \ref{main} is 
obtained as a limit of smooth solutions 
$u^{K}=u^{K}(t,x)$
with a parameter $K>0$ corresponding to the truncated initial data
\eq{\label{eq;truncated-initial-u}
 u_{0}^{K}(x)
 \equiv
 \left\{
 \spl{
  &\ep \frac{d}{\sg_{d}}K^{2} ,\quad && |x| \le R(K),\\
  &\ep \frac{2(d-2)}{|x|^{2}} ,\quad && R(K) \le |x|,
 }
 \right.
}
where 
$$
 R(K)
 \equiv
 \left[
  \frac{2(d-2)\sg_{d}}{d}
 \right]^{\frac12}
 K^{-1}.
$$
Let us state properties of these approximating initial conditions as well as the corresponding solutions to system \eqref{eq;DD}.

\begin{lemma}\label{lem;existence-truncated-initial}
 For every $K>0$ and $\varepsilon \in (0,1)$, 
 there exists a nonnegative, radially symmetric 
 and smooth global-in-time solution $u^{K}=u^{K}(t,x)$  to system \eqref{eq;DD}
 with the initial datum $u_{0}^{K}=u_{0}^{K}(x)$.
\end{lemma}

\begin{proof} 
Since the initial datum 
$ 
 u_{0}^{K}
$
is bounded and 
satisfies $u_{0}^{K}(x) \le \ep u_{C}(x)$ for $x \in \re^{n}$, we have
$
 u_{0}^{K} 
 \in
 M^{d/2}(\re^{n}) 
 \cap M^{p}(\re^{n}) 
$ 
for each $p>d/2$ 
together with the estimate 
$$
 \xn u_{0}^{K} \xn
 \le
 \xn \ep u_{C} \xn
 =\ep 2\sg_{d}
 <2\sg_{d}.
$$
Thus, by Theorem \ref{glo}, the corresponding solution 
$
u^{K}
$ to problem
\eqref{eq;DD} is nonnegative, radially symmetric and smoooth global-in-time solution.
\end{proof}

For the solutions and their initial data from Lemma \ref{lem;existence-truncated-initial}, we define
the radial mass distribution functions \eq{\label{eq;integrated-density-approximation}
 M^{K}(t,r)
 \equiv
 \int_{\{|x|<r\}} 
u^{K}(t,x) \, \dx
 \quad
 \text{\rm and}
 \quad
 M^{K}_{0}(r)
 \equiv
 \int_{\{|x|<r\}} 
u^{K}_{0}(x) \, \dx
}
for $t>0$ and $r>0$,
and we study their properties.

\begin{lemma}\label{lem;local-sol-property} 
Let $M^{K}=M^{K}(t,r)$ be the radial mass distribution function of the solution obtained in Lemma \ref{lem;existence-truncated-initial}.
Then   $M^{K}$ is the unique solution to problem \eqref{eq;local-mass} with the initial data $M^{K}_{0}$ and  
\eq{\label{eq;apriori-M^{K}}
 M^{K}(t,r)
 \le
  \ep 2\sg_{d} r^{d-2} \qquad \text{for all} \quad t>0, \quad r>0.
}
\end{lemma}

\begin{proof}
By Lemma \ref{lem;existence-truncated-initial}, Proposition \ref{prop;?} 
and  Theorem \ref{comp},
the radial mass distribution function 
$
M^{K}
$ 
is the unique solution to \eqref{eq;local-mass} with the initial data $M_{0}^{K}(r)$
and
satisfies property (3) of Proposition \ref{prop;?}. 
\end{proof}

\begin{lemma}\label{lem;self-similar-prop-approximation}
 For each $K>0$ and  $\lam > 0$,
 we have the following scaling property
 \begin{equation}\label{MKl}
  \lam^{2-d}M^{K}(\lam^{2} t,\lam r)
  =M^{K\lam}(t,r)
  \qquad \text{for all} \quad t>0,\;r>0.
 \end{equation}
\end{lemma}
\begin{proof} 
We define the rescaled solution 
$
 u_{\lam}^{K}
 =
 u_{\lam}^{K}(t,x)
$ 
and the rescaled initial datum 
$
 u_{0,\lam}^{K}
 =
 u_{0,\lam}^{K}(x)
$
as follows
\eq{\label{eq;scaled-app-sol}
 u_{\lam}^{K}(t,x)
 \equiv
 \lam^{2}u^{K}(\lam^{2}t,\lam x)
 \quad
 \text{\rm and}
 \quad
 u_{0,\lam}^{K}(x)
 \equiv
 \lam^{2}u_{0}^{K}(\lam x).
} 
Obviously,
it holds that
\eqn{
 u_{0,\lam}^{K}(x) 
 =
 \left\{
 \spl{
  &\ep(K\lam)^{2}\frac{d}{\sg_{d}} ,\quad && |x| \le {R(K\lam)}\\
  &\ep \frac{2(d-2)}{|x|^{2}} ,\quad && {R(K\lam)} \le |x|
 }
 \right\} =u_0^{K\lam}(x).
}
Then we introduce rescaled radial mass distribution functions 
$$
 M_{\lam}^{K}(t,r)
 \equiv\lam^{2-d}M^K(\lam ^{2}t,\lam r)
 \quad
 \text{\rm and}
 \quad
 M_{0,\lam}^{K}(r)
 \equiv\lam^{2-d}M_{0}^{K}(\lam r)
$$
and, accordingly, we have
\eqn{
 \spl{
  M_{0,\lam}^{K}(r)
  =
  \lam^{2-d}
  \int_{\{|x|<\lam r\}} 
   u_{0}^{K}(x) 
  \, \dx
  =
  \int_{\{|x|<r\}}
   \lam^{2}u_{0}^{K}(\lam y) 
  \,\dy
  =
  \int_{\{|x|<r\}}
   u_{0}^{K\lam}(y) 
  \,\dy
  =
  M_{0}^{K\lam}(r).
 }
}
Thus, by 
the uniqueness of solutions from Lemma \ref{lem;local-sol-property}, 
we obtain the following equality for each $K>0$  and $\lam>0$
\be\label{uKlam}
M_{\lam}^{K}(t,r)=M^{K\lam}(t,x)\qquad \text{for all}\quad t>0, \ \ x\in \R^{d},
\ee
which implies our desired identity \eqref{MKl}.
 \end{proof}

Furthermore, the comparison principle from Theorem \ref{comp} shows that 
the sequence
$
 \{ 
  M^{K}(t,r) 
 \}_{K>0}
$ 
is increasing monotonically with respect to $K>0$.

\begin{lemma}\label{lem;monotonicity-approximation}
 Let $\{ M^{K}(t,r) \}_{K>0}$ be a sequence of solutions to problem
 \eqref{eq;local-mass} with the initial data  
 $M_{0}^{K}=M_{0}^{K}(r)$.
 Then, for $0<K_{1} \le K_{2}$, 
 $$
  M^{K_{1}}(t,r) \le M^{K_{2}}(t,r)
  \qquad
  \text{for all}
  \quad t>0, \quad r>0.
 $$
\end{lemma}
\begin{proof} 
First, we observe that $u_{0}^{K_{1}}(x) \le u_{0}^{K_{2}}(x)$
for $K_{1} \le K_{2}$ by  definition of $u_{0}^{K}$ in  \eqref{eq;truncated-initial-u}.
According to Proposition \ref{prop;?}, the
functions $M^{K_{1}}, M^{K_{2}}$ satisfy assumptions \eqref{subcrit}--\eqref{bc} of Theorem \ref{comp}.
Moreover, since the  estimate \eqref{eq;apriori-M^{K}} is valid for all $K>0$, we have
$$
 |M^{K_{1}}(t,r)-M^{K_{2}}(t,r)|
 \le
 \ep 4 \sg_{d} r^{d-2}.
$$
Thus, the  uniform estimate \eqref{diff} also holds when
$(\underline{M},\overline{M})=(M^{K_{1}},M^{K_{2}})$.
Applying Theorem~\ref{comp}
we conclude 
$$
 M^{K_{1}}(t,r)
 \le
 M^{K_{2}}(t,r)
 \qquad
\text{for all} 
\quad 
t\ge 0,\quad
r \ge  0.
$$ 
\end{proof}

\begin{lemma}\label{lem;limit}
 Let $\{ M^{K}(t,r) \}_{K>0}$ be a sequence of solutions to problem
 \eqref{eq;local-mass} with the initial data  
 $M_{0}^{K}=M_{0}^{K}(r)$.
 Then
 \begin{equation}
 \label{eq;definition-M}
 \lim_{K\to \infty} M^K(t,r) \equiv M^*(t,r)\leq \ep 2 \sg_{d} r^{d-2}
 \qquad \text{for all} \quad t\geq 0, \; r\geq 0.
 \end{equation}
 The limit function 
$M^{*}$ has the self-similar property 
\eq{\label{eq;self-similar-prop}
 \lam^{2-d}M^*(\lam^{2}t,\lam r) = M^*(t,r)
 \qquad \text{for all} \quad  t\geq 0, \; r\geq 0
}
for each  $\lam > 0$.
\end{lemma}

\begin{proof}
The relations in \eqref{eq;definition-M} result 
from Lemma \ref{lem;monotonicity-approximation} by which 
the sequence $\{ M^{K}(t,r)\}_{K>0}$ 
is monotonically increasing in $K>0$ and, from Lemma \ref{lem;local-sol-property}, where we obtain \eq{\label{eq;upper-bound-approximation}
 0
 \le
 M^{K}(t,r)
 \le
\ep 2\sg_{d}r^{d-2}.
}%
Since the family of the solutions 
$\{ M^{K\lam}(t,r) \}_{K>0}$
coincides with the family  $\{ M^{K}(t,r)\}_{K>0}$, 
we observe 
$ 
 \lim_{K \to \infty}
 M^{K\lam}(t,r)
 =
 M^*(t,r)
$
by relation \eqref{eq;definition-M}.
Thus, using Lemma \ref{lem;self-similar-prop-approximation}
we obtain
$$
 \lam^{2-d}M^*(\lam^{2}t,\lam r) 
 =
 \lim_{K \to \infty}
 \lam^{2-d}M^{K}(\lam^{2}t,\lam r)
 =
 \lim_{K \to \infty}
 M^{K\lam}(t,r)
 =
 M^*(t,r)
$$
for each $\lambda\ge 0$ and all $t\geq 0$, $r\geq 0$.
\end{proof}

In the next step, we prove suitable Schauder estimates for the family 
$\{ M^{\rd K}(t,r) \}_{K>0}$
in order to show that the limit function $M^*(t,r)$ is a classical solution of problem \eqref{eq;local-mass}. Here, we follow an approach proposed in 
\cite{BKLN1,BKLN2}.

\begin{theorem}\label{thm;prop-of-M-limit}
For every  $0<\tau<T$, $0<\delta<R$
 and
$\alpha\in(0,1)$, 
it holds that   
$M^{*} \in \mathcal{C}^{1+\frac{\alpha}{2},2+\alpha}_{t,r}([\tau,T]\times [\delta,R])$
and 
\eq{\label{eq;limit-M}
 \lim_{K \to \infty}
 \| M^{K} - M^{*}\|_{\mathcal{C}^{1+\frac{\alpha}{2},2+\alpha}_{t,r}([\tau,T]\times [\delta,R])}
 = 0.
}%
Moreover, the function 
$
 M^{*} 
 \in 
 \mathcal{C}_{t,r,\text{\rm loc}}^{1+\frac{\alpha}{2},2+\alpha}((0,\infty) \times (0,\infty))\cap C([0,\infty) \times [0,\infty))
$ 
is a classical solution to problem \eqref{eq;local-mass}
with the initial datum $M_{0}(r)=\ep 2\sg_{d}r^{d-2}$.
\end{theorem}

\begin{proof} 
By the standard parabolic regularity argument (\cite{LSU}),
there exists a constant ${C}_{0}(\alpha,\tau,\delta,R,T)>0$,  independent of $K>0$, such that
\eq{\label{eq;schauder-est-M-approximation}
 \| M^{K} \|_{{\mathcal C}^{1+\frac{\alpha}{2},2+\alpha}_{t,r}([\tau,T]\times [\delta,R])}
 \le {C}_{0}(\alpha,\tau,\delta,R,T).
}%
Estimate \eqref{eq;schauder-est-M-approximation} combined with 
the Arzel\`{a}--Ascoli theorem and the definition of $M^{*}$ in 
 \eqref{eq;definition-M} imply that ${M^{*}} \in {\mathcal C}^{1+\frac{\alpha}{2},2+\alpha}_{t,r}([\tau,T]\times [\delta,R])$
 and 
\begin{equation}\label{eq;limit-M:H}
 \lim_{K \to \infty} 
 \| 
  M^{K} - M^{*} 
 \|_{{\mathcal C}^{1+\frac{\alpha}{2},2+\alpha}_{t,r}([\tau,T]\times [\delta,R])}
 = 0.
\end{equation} 
Since all functions  $M^{K}$ are  unique classical solutions to problem \eqref{eq;local-mass}
with the initial data $M_{0}^{K}$ and they converge  to $M^{*}$ as $K\to\infty$
in the sense of limit \eqref{eq;limit-M:H},
we obtain that the limit function $M^{*}$ is a classical solution to problem
\eqref{eq;local-mass} with the initial datum $M_{0}(r)=\ep 2 \sg_{d}r^{d-2}$.
\end{proof}

We conclude this section by a result on a self-similar asymptotics of some solutions to problem \eqref{eq;local-mass}.

\begin{corollary}\label{cor;self-similar-asymptotics} 
Assume that a continuous  and nondecreasing initial datum $M_{0}$ satisfies 
$$
 M^{K_{1}}_0(r)\le M_0(r)\le M^{K_{2}}_0(r)
 \qquad
 \text{for all }
 \quad
 r>0, 
$$
 some $K_{1}<K_{2}$, with   $M_0^{K_i}$ $(i\in \{1,2\})$ given by formulas
\rf{eq;truncated-initial-u} and 
 \rf{eq;integrated-density-approximation}. 
Then, the corresponding solution $M(t,r)$ 
of problem \eqref{eq;local-mass}
has  self-similar asymptotics, namely,
\eq{\label{eq;self-similar-asymptotics}
 \lim_{\lam \to \infty}
 \lam^{2-d} M(\lam^{2}t,\lam r) 
 = 
 M^{*}(t,r),
}%
where  $M^{*}=M^{*}(t,r)$ is the self-similar solution of problem \eqref{eq;local-mass}, and the convergence is uniform on compact subsets of $(0,\infty)\times (0,\infty)$.
\end{corollary}

\begin{proof}
Applying the comparison principle from Theorem \ref{comp} 
we obtain
$$
 M^{K_{1}}(t,r) 
 \le 
 M(t,r) 
 \le
  M^{K_{2}}(t,r)
  \qquad
  \text{for all}
  \quad
  t>0,
  \;
  r>0.
$$ 
Therefore, for every  $\lam > 0$, we have
$$
 \lam^{2-d}M^{K_{1}}(\lam^{2} t,\lam r)
 \le 
 \lam^{2-d}M(\lam^{2}t,\lam r) 
 \le 
 \lam^{2-d}M^{K_{2}}(\lam^{2} t,\lam r).
$$
On the other hand, the scaling property from 
Lemma \ref{lem;self-similar-prop-approximation} implies 
$$
 \lam^{2-d}M^{K_{1}}(\lam^{2} t,\lam r)=M^{K_{1}\lam}(t,r)
 \quad
 \text{and}
 \quad
 \lam^{2-d}M^{K_{2}}(\lam^{2} t,\lam r)=M^{K_{2}\lam}(t,r).
$$
Since both  families
$\{ M^{K_{1}\lam}(t,r) \}_{\lam > 0}$ 
and 
$\{ M^{K_{2}\lam}(t,r) \}_{\lam > 0}$ 
coincide with 
$\{ M^{K}(t,r) \}_{K > 0}$,
we conclude by Theorem \ref{thm;prop-of-M-limit}
$$
 \lim_{\lam \to \infty} M^{K_{1}\lam}(t,r)
 =
 \lim_{\lam \to \infty} M^{K_{2}\lam}(t,r)
 =M^{*}(t,r)
$$
uniformly on compact subsets of $(0,\infty)\times (0,\infty)$. 
\end{proof}

\begin{remark}\label{rem:fin}
For the reader convenience, let us review results which 
we have already proved in this section in terms of solutions to system \eqref{eq;DD}. Applying  formula \eqref{u-M} to the self-similar solution in Theorem \ref{thm;prop-of-M-limit}, we obtain the function
\be 
u^*(t,x)=\frac{1}{\s_dr^{d-1}}M^*_r(t,r) \qquad \text{with} \quad  r=|x|>0.\label{u-M*}
\ee 
By the scaling property \eqref{eq;self-similar-prop}, we have got  the equality $\lambda^2 u^*(\lambda^2t,\lambda x)=u^*(t,x)$, hence (see equations \eqref{u:ul}--\eqref{eq;self-similar}), we obtain 
\be \label{u*}
u^*(t,x) =\frac{1}{t}U\left(\frac{x}{\sqrt{t}}\right).
\ee
By Theorem \ref{thm;prop-of-M-limit}, the self-similar profile $U(x)=u^*(1,x)$, together with its derivatives up to second order are H\"{o}lder continuous on 
$\re^d \setminus \{ 0 \}$ (by a standard parabolic regularity, this is in fact a smooth function on  $\re^d\setminus \{0\}$). In the next section, we prove that   
$u^*(t,x)$ given by formulas \eqref{u-M*}--\eqref{u*} 
is a self-similar solution of system \eqref{eq;DD} with the self-similar profile $U\in C^\infty (\R^d)$. 
\end{remark}

\section{Regularity of self-similar profile} \label{sec:reg}

In order to study the regularity of $M^{*}=M^*(t,r)$ as $r\to 0$, and then the regularity of the corresponding density $u^{*}(t,x)$ (see Remark \ref{rem:fin}), 
 we introduce the following  auxiliary linear initial-boundary value problem.

\begin{lemma}\label{lem;Bessel-semigroup}
The following initial-boundary value problem on the half-line
\eq{\label{eq;Bessel} 
 \spl{
  & m_t = m_{rr}-\frac{\lambda}{r}m_r,\quad & t>0, r > 0,\\
  & m(t,0)= 0,\quad &t>0,\\
  & m(0,r) = c_{0} r^{d-2} , \quad &r> 0,
 }
}%
with some $\lambda\in(d-3,d-1]$,
has  the unique solution of the following explicit form 
\eq{\label{eq;m-explicit-form}
 m(t,r)
 =
 \frac{2^{d-3-\lam}}{\Gamma\big( \frac{\lam-d+3}{2}\big)}
 c_{0}
 t^{-\frac{\lam-d+3}2}
 r^{\lam+1}
 {\rm e}^{-\frac{r^{2}}{4t}}
 \int_{0}^{1}
  s^{\frac{d}2-1}(1-s)^{\frac{\lam-d+3}2-1}{\rm e}^{\frac{r^{2}}{4t}s}
 \,
 \ds,}%
 where $\Gamma=\Gamma(x)$ is the Euler Gamma function. 
Moreover, for each $\y \in (0,\infty)$, this solution satisfies 
\eq{\label{eq;mmm} 
 \sup_{t>0}
 \int_{0}^{\sqrt{t}\y}
  \frac{m(t,r)}{r^{d-1}}
 \,
 \dr
 <
 \infty.
} 
\end{lemma}

\begin{proof}
By,   e.g.,   either  \cite[App. 1, 21. Bessel processes, p. 138]{BS} or \cite{NSZ},
the solution to problem \eqref{eq;Bessel} is explicitly given by the formula 
\eq{\label{eq;msol}
 m(t,r)
 =
 \int_0^\infty 
  p(t;r,s)m_0(s)s^{-\lambda}
 \ds
}%
with $m_{0}(s)=c_{0}s^{d-2}$ and 
with the kernel 
\eq{\label{eq;Bkernel}
p(t;r,s)=\frac{1}{2t}{\rm e}^{-\frac{r^2+s^2}{4t}}(rs)^\frac{\lambda+1}{2} I_\frac{\lambda+1}{2}\left(\frac{r}{2t}s\right). 
}
Here,  the function $I_{\nu}=I_{\nu}(x)$ is the modified Bessel function of the first kind 
$$
 I_\nu(x)
 =
 \sum_{m=0}^{\infty}
  \frac{1}{m! \Gamma(m+\nu+1)} 
  \left(
   \frac{x}2
  \right)^{2m+\nu}.
$$  
Integral \rf{eq;msol} which defines the solution $m=m(t,r)$ of 
problem \eqref{eq;Bessel} with $m_0(s)=c_{0}s^{d-2}$ can be expressed in terms of the confluent hypergeometric function ${}_{1}F_{1}$, 
see \cite[Lemma 2.2]{NS}. 
Indeed, by \cite[Lemma 2.2]{NS}, 
the identity 
\eq{\label{eq;NS-Prudnikov}
\int_0^\infty s^{\beta-1}{\rm e}^{-ps^2}I_\nu\left(qs\right)\ds=\frac{q^\nu}{2^{\nu+1}p^\frac{\beta+\nu}{2}}\frac{\Gamma\left(\frac{\beta+\nu}{2}\right)}{\Gamma(\nu+1)}\,{}_1F_1\left(\frac{\beta+\nu}{2};\nu+1;\frac{q^2}{4p}\right)
}%
holds for $p$, $q>0$, $\beta,\, \nu\in\R$, $\beta+\nu>0$, $\nu \ne -1,\, -2,\, \cdots$,
where
$$
 {}_{1}F_{1}
  \left(
   a ; b ; z
  \right)
 =
 \frac{\Gamma(b)}{\Gamma(a)\Gamma(b-a)}
 \int_{0}^{1}
  s^{a-1}(1-s)^{b-a-1}{\rm e}^{zs}
 \,
 \ds
$$
with $0 < a < b$.
Thus, when $m_{0}(s)=c_{0}s^{d-2}$, we observe
\eqn{
 \spl{
  \int_0^\infty 
   p(t;r,s)m_0(s)s^{-\lambda}
  \ds
  =&
  \frac{c_{0}}2
  t^{-1}
  r^{\frac{\lam+1}2}
  {\rm e}^{-\frac{r^{2}}{4t}}
  \int_{0}^{\infty}
   s^{d-\frac{\lam+1}{2}-1}
   {\rm e}^{-\frac1{4t}s^{2}}
   I_{\frac{\lam+1}2}
   \left(
     \frac{r}{2t}
    s
   \right)
  \,
  \ds\\
  =&
  \frac{c_{0}}2
  t^{-1}
  r^{\frac{\lam+1}2}
  {\rm e}^{-\frac{r^{2}}{4t}} \,
  \frac{( \frac{r}{2t} )^{\frac{\lam+1}2}}{2^{\frac{\lam+3}2}(\frac1{4t})^{\frac{d}2}}
  \frac{\Gamma(\frac{d}2)}{\Gamma(\frac{\lam+3}2)} \, 
  {}_{1}F_{1}
   \left(
    \frac{d}{2};\frac{\lam+3}{2};\frac{r^2}{4t}
   \right)\\
  =&
  \frac{2^{d-3-\lam}}{\Gamma \left( \frac{\lam -d+3}{2} \right)}
  c_{0}
  t^{-\frac{\lam-d+3}{2}}
  r^{\lam + 1}
  {\rm e}^{-\frac{r^{2}}{4t}}
  \int_{0}^{1}
   s^{\frac{d}2-1}(1-s)^{\frac{\lam-d+3}2-1}{\rm e}^{\frac{r^{2}}{4t}s}
  \,
  \ds,
 }
}%
which implies our desired conclusion \eqref{eq;m-explicit-form}.
We next prove uniform bound \eqref{eq;mmm}.
Since $\lam-d+3 > 0$
and ${\rm e}^{\frac{r^{2}}{4t}s} \le {\rm e}^{\frac{r^{2}}{4t}}$
for $0 \le s \le 1$,
we observe 
\eqn{
 \spl{
  \int_{0}^{\sqrt{t}\y}
   \frac{m(t,r)}{r^{d-1}}
  \, 
  \dr
 =&
  Cc_{0}
  t^{-\frac{\lam-d+3}{2}}
  \int_{0}^{\sqrt{t}\y}
  r^{\lam-d+2}
  {\rm e}^{-\frac{r^{2}}{4t}}
  \left(
  \int_{0}^{1}
   s^{\frac{d}2-1}(1-s)^{\frac{\lam-d+3}2-1}{\rm e}^{\frac{r^{2}}{4t}s}
  \,
  \ds
  \right)
  \dr\\
 \le &
  Cc_{0}
  t^{-\frac{\lam-d+3}{2}}
  \int_{0}^{\sqrt{t}\y}
  r^{\lam -d+2}
  \left(
  \int_{0}^{1}
   s^{\frac{d}2-1}
   (1-s)^{\frac{\lam-d+3}2-1}
  \,
  \ds
  \right)
  \dr\\
 = &
  Cc_{0}
  B
  \left(
  \frac{d}{2},\frac{\lam-d+3}2
  \right)
  t^{-\frac{\lam-d+3}{2}}
  \int_{0}^{\sqrt{t}\y}
  r^{\lam -d+2}
  \dr\\
 = &
  \frac{Cc_{0}}{\lam -d+3}
  B
  \left(
  \frac{d}{2},\frac{\lam-d+3}2
  \right)\y^{\lam-d+3}
  <\infty,
 }
}
where $\y \in (0,\infty)$, $C=C(\lam,d)>0$ and $B=B(x,y)$ is the Euler Beta function. 
\end{proof} 

\begin{remark}\label{rem;self-similarity-of-m}
 Thanks to the self-similar property of the solution 
 $m=m(t,r)$ to linear problem   \eqref{eq;Bessel},
 the quantity 
 $$
  g_{\y}(t)
  \equiv
   \int_{0}^{\sqrt{t}y_{*}} 
    \frac{m(t,r)}{r^{d-1}} \, 
   \dr
 $$
 is finite, and does not depend on $t>0$ when $m_{0}(r)=c_{0}r^{d-2}$.
 Indeed,
 for each $\mu >0$,
 we observe
\eqn{
 \spl{
  g_{\y}(\mu^{2}t)
  =
  \int_{0}^{\sqrt{t}\mu\y}
   \frac{m(\mu^{2}t,r)}{r^{d-1}} \, 
  \dr
  =
  \int_{0}^{\sqrt{t}\y}
   \frac{\mu^{2-d}m(\mu^{2}t,\mu r)}{r^{d-1}} \, 
  \dr
  =
  \int_{0}^{\sqrt{t}\y}
   \frac{m(t,r)}{r^{d-1}} \, 
  \dr
  =g_{\y}(t)
 }
}
for every $\mu>0$ and $t > 0$, 
since $\mu^{2-d}m(\mu^{2}t,\mu r)=m(t,r)$ holds for every $\mu>0$, 
$t \ge 0$ and $r \ge 0$.
It is easy to check the self-similar property of the function $m$ from
the explicit form \eqref{eq;m-explicit-form}.
We also mention the problem of the critical case $\ep = 1$. 
Since $\lam=d-1-2\ep$, we observe 
$
 \lam + 2 - d 
 =
 -1
$,
and thus 
$
r^{\lam + 2 - d}=r^{-1}
$ 
is not integrable near $r=0$.
Therefore the uniform bound \eqref{eq;mmm} does not hold.
\end{remark}

The solution $m=m(t,r)$ to problem \eqref{eq;Bessel} plays the role of a barrier function of 
our original solution $M=M(t,r)$ to problem \eqref{eq;local-mass}.

\begin{lemma}\label{lem;super-sol-msol}
Let $M=M(t,r)$ be a solution to the initial-boundary value problem \eqref{eq;local-mass} with $M_{r}(t,r) \ge 0$ and $M(t,r) \le \ep 2\sg_{d}r^{d-2}$ for
$t\ge 0, r \ge 0$ and $\ep \in (0,1)$.
Assume that the function $\overline{m}=\overline{m}(t,r)$ 
is the solution to linear problem \eqref{eq;Bessel}
with $c_0=\eps 2\s_d$ and 
$\lam=d-1-2\ep$. 
Moreover, let 
$\underline{m}=\underline{m}(t,r)$ be also the solution to problem \eqref{eq;Bessel}
with $c_0=\eps 2\s_d$ and $\lam=d-1$.
Then,  $\underline{m}(t,r) \le M(t,r) \le \overline{m}(t,r)$ holds for $t>0$ and $r \ge 0$.
\end{lemma}

\begin{proof}
Since the inequality 
$ 
 M(t,r)
 \le 
  \eps 2\sg_dr^{d-2}
$
holds for those solutions and $M_r(t,r)\ge 0$, 
equation \rf{eq;local-mass} leads to the inequality 
\eqn{ 
 M_t
 \le 
 M_{rr}-\frac{d-1}{r}M_r+\frac{1}{\sg_dr^{d-1}}\ep 2\sg_{d}r^{d-2}M_r 
 =
 M_{rr}-\frac{d-1-2\eps}{r}M_r.
}
Therefore, the solution $M$ to problem \eqref{eq;local-mass} 
corresponding the initial datum $M_{0}(r)=c_{0}r^{d-2}$
with $M_{r}(t,r) \ge 0$ and $M(t,r) \le \ep 2\sg_{d}r^{d-2}$
is a subsolution to linear problem \eqref{eq;Bessel}. 
Hence, by the comparison principle, 
we conclude that 
$M(t,r) \le \overline{m}(t,r)$ for
$t \ge 0$ and $r \ge 0$. 

We next show that the function $M$ is a supersolution to linear problem \eqref{eq;Bessel}.
According to the conditions $M \ge 0$ and $M_{r} \ge 0$,
we observe 
\eqn{ 
 M_t
 =
 M_{rr}-\frac{d-1}{r}M_r
 +
 \frac{1}{\sg_{d}r^{d-1}}MM_{r}
 \ge 
 M_{rr}-\frac{d-1}{r}M_r.
}
Thus the function $M$ is a supersolution to problem \eqref{eq;Bessel}.
Again by the comparison principle for linear problem \eqref{eq;Bessel},
we have $\underline{m}(t,r) \le M(t,r)$ for $t \ge 0$ and $r \ge 0$.
\end{proof}

Finally, we prove that the density $u^{*}$   corresponding to the self-similar solution $M^{*}$ of problem \eqref{eq;local-mass}
is of the form $u^{*}(t,x)=\frac{1}{\sg_{d}|x|^{d-1}}M^{*}_{r}(t,|x|) =\frac{1}{t}U\big(\frac{x}{\sqrt{t}}\big)$ with some bounded function $U$: $\|U\|_\infty<\infty$. 

\begin{lemma}\label{lem;profile-boundedness}
Let $\ep \in (0,1)$,
$M^{*}=M^{*}(t,r)$ be a solution to problem \eqref{eq;local-mass}
constructed in Theorem \ref{thm;prop-of-M-limit} with $M_{0}(r)=\ep 2\sg_{d}r^{d-2}$
and 
$u^{*}=u^{*}(t,x)$ be a function defined by 
$u^{*}(t,x)=\frac{1}{\sg_{d}|x|^{d-1}}M^{*}_{r}(t,|x|)$. 
Then,  
$u^{*}$ is of the form $u^{*}=\frac{1}{t}U\big(\frac{x}{\sqrt{t}}\big)$ with some $U \in L^{\infty}(0,\infty)$. 
\end{lemma}

\begin{proof}
The function $M^{*}$, by property \rf{eq;self-similar-prop}, has the form 
\be \label{eq;M-sss} 
  M^{*}(t,r)
  \equiv
  t^{\frac{d}2-1}\M \left(\frac{r}{\sqrt{t}}\right),
\ee 
On the other hand, 
$$
 M^{*}(t,r)
  =
  \int_{\{ |x|<r \}}
   u^{*}(t,x)
  \,
  \dx, 
$$ 
thus
\eq{\label{eq;M-sss-2}
 M^{*}(t,r)
 = 
 t^{\frac{d}2-1} 
 \int_{\{ |x| < \frac{r}{\sqrt{t}}\}}
  U(x)\,
 \dx
}
holds for a function $U=U(y)$ such that $
 u^{*}(t,x)
 =
 \frac1t 
 U\big( \frac{x}{\sqrt{t}}\big)
$.
   
The function $\M =\M(y)$ satisfies the equation 
\be
\M''+\frac{y}{2}\M'-\frac{d-2}{2}\M -\frac{d-1}{y}\M'+\frac{1}{\s_dy^{d-1}}\M \M'=0,\label{M-sss}
\ee 
with $\M(0)=0$ and the new variable $y=\frac{r}{\sqrt{t}}$, $'=\frac{\rm d}{{\rm d}y}$. 
Therefore, the function $U$ satisfies 
\eq{\label{u-sss}
U'(y)+\left(\frac{y}{2}+\frac{\M(y)}{\s_dy^{d-1}}\right)U(y)
=
\frac{d-2}{2}\frac{\M(y)}{\s_dy^{d-1}}.
}
According to Lemma \ref{lem;super-sol-msol} 
 and property \eqref{eq;mmm} in Lemma \ref{lem;Bessel-semigroup},
we get 
\eqn{ 
 \spl{
  \int_0^{\y}
   \frac{\M (y)}{y^{d-1}}
  \,
  \dy
  =
  \int_{0}^{y_{*}}
   \frac{M^{*}(1,y)}{y^{d-1}}
  \,
  \dy
  \le
  \int_{0}^{y_{*}}
   \frac{\overline{m}(1,y)}{y^{d-1}}
  \,
  \dy
  \le
  C(\lam, d)
  \y^{\lam+3-d}
  <\infty
} } 
for some $y_{*} \in (0,\infty)$. 
Now suppose that $U(\y)$ is finite for some (in fact,  by Lemma \ref{lem;super-sol-msol}, for each) $\y \in (0,\infty)$. 
We may introduce the integrating factor
\eq{\label{eq;integral-factor-f}
 f(y)
 =
 \exp\left(\frac{y^2}{4}-\int_y^{\y} \frac{\M(s)}{\s_ds^{d-1}}\ds\right),
}
and  thus
we rewrite equation \rf{u-sss} in the form 
\be
(Uf)'=U'f+f'U=\frac{d-2}{2}\left(f'-\frac{y}{2}f\right). \label{U-int}
\ee
Hence, by property \rf{eq;mmm}, $0<f(0)<\infty$ holds together with $f(\y)={\rm e}^{\frac14 \y^2}$.
Integrating equation \rf{U-int} over $[y,\y]$, we conclude  
\eq{\label{eq;U-explicit-form}
 U(y)- \frac{d-2}2
 =
\frac{f(\y)}{f(y)}\left(U(\y)-\frac{d-2}2\right) 
 +\frac{d-2}{4f(y)}\int_y^{\y} sf(s) \, \ds,
}
which is finite and bounded as $y \searrow 0$ since $U(\y)<\infty$.
We claim that, for $\y>0$  and $y=0$,
\eq{\label{eq;limit-U(0)-rhs}
\frac{\mathrm{d}}{\dy_{*}}
\left[
\frac{f(\y)}{f(0)}\left(U(\y)-\frac{d-2}2\right) 
 +\frac{d-2}{4f(0)}\int_0^{\y} sf(s) \, \ds
\right]
 =0,
}
which shows that the right hand side of equality \eqref{eq;U-explicit-form} is independent of $\y>0$.
To see this, we substitute
the formula
$$
 \frac{\mathrm{d}}{\dy_{*}}
 \int_0^{\y} sf(s) \, \ds
 =
 \y f(\y) 
 -
 \frac{\mathcal{M}(\y)}{\sg_{d}\y^{d-1}}
 \int_{0}^{\y}
  sf(s)
 \, \ds
$$
and relation \eqref{u-sss} into the left hand side of \eqref{eq;limit-U(0)-rhs}.
Then we immediately obtain our claim \eqref{eq;limit-U(0)-rhs}.
Therefore $\lim_{y \searrow 0} U(y)$ exists and that limit is independent of $\y>0$.

We note, moreover, that $\lim_{y \searrow 0}U(y) >0$.
Indeed, since $\lim_{y \searrow 0}U(y)$ exists, the de l'Hospital rule shows that
\eq{\label{eq;limit-relation-U}
  \frac{\sg_{d}}{d}
  \lim_{y \searrow 0}
  U(y)
  =
  \sg_{d}
  \lim_{y \searrow 0}
  \frac{U(y)y^{d-1}}{dy^{d-1}}
  =
  \sg_{d}
  \lim_{y \searrow 0}
  \frac{\int_{0}^{y}U(s)s^{d-1}\, \ds}{y^{d}}
  =
 \lim_{y \searrow 0 }
  \frac{\mathcal{M}(y)}{y^{d}}
 .
}
Let $\underline{m}=\underline{m}(t,r)$ solve problem \eqref{eq;Bessel} with the initial datum $m_{0}(r)=\ep 2\sg_{d}r^{d-2}$ and $\lam = d-1$. Solutions of this problem \eqref{eq;Bessel} are the integrated radial solutions of the usual heat equation, so that $\underline{m}(t,r)\le M(t,r)$, as in Lemma \ref{lem;super-sol-msol}. 
According to formula \eqref{eq;m-explicit-form} of Lemma \ref{lem;Bessel-semigroup}, 
$\underline{m}$ has the form  
$$
 \underline{m}(t,r)
 =
 \ep\frac{\sg_{d}}{2}
  t^{-1}
  r^{d}
  {\rm e}^{-\frac{r^{2}}{4t}}
  \int_{0}^{1}
   s^{\frac{d}2-1}
   {\rm e}^{\frac{r^{2}}{4t}s}
  \,
  \ds.
$$ 
Then we have
$$
 t^{\frac{d}2-1}\underline{m}
 \bigg(
 1,\frac{r}{\sqrt{t}}
 \bigg)
 =
 \underline{m}(t,r)
 \le
 M^{*}(t,r)
 =
 t^{\frac{d}2-1}
 \mathcal{M}
 \bigg(
  \frac{r}{\sqrt{t}}
 \bigg)
$$
or equivalently 
\eqn{ 
 \underline{m}(1,y)
 \le 
 \mathcal{M}(y)
}
for $y>0$,
and thus
\eqn{ 
 \spl{
 \frac{\mathcal{M}(y)}{y^{d}}
 \ge
 \frac{\underline{m}(1,y)}{y^{d}}
 = 
 \ep\frac{\sg_{d}}{2}
 {\rm e}^{-\frac14 y^{2}}
 \int_{0}^{1}
  s^{\frac{d}2-1}{\rm e}^{\frac14 y^{2} s}
 \, \ds.
 }
}
Since ${\rm e}^{-\frac14 y^{2}} \to 1$ 
and 
$
 \int_{0}^{1}
  s^{\frac{d}2-1}{\rm e}^{\frac14 y^{2} s}
 \, \ds
 \to 
 \int_{0}^{1}
  s^{\frac{d}2-1}
 \, \ds
 =
 \frac2d
$
as $y \searrow 0$,
we obtain
\eq{\label{eq;lower-bound-profile-limit}
 \lim_{y \searrow 0} 
 \frac{\mathcal{M}(y)}{y^{d}}
 \ge
 \frac{\sg_{d}}{d} \ep 
 .
}
Relation \eqref{eq;limit-relation-U} 
and the lower bound \eqref{eq;lower-bound-profile-limit} finally show that
$$
 \lim_{y \searrow 0}
 U(y)
 \ge 
 \ep
 >0.
$$
This fact indicates that $U(y) > 0$ for $y \ge 0$. Indeed,
$$
 0
 <
 f(0)
 \lim_{z \searrow 0} U(z)
 =
 \lim_{z \searrow 0}
 f(z)U(z) 
 \le
 f(y_{*})U(y_{*})
 \le
 f(y)U(y)
$$
for $y \ge y_{*}$ since
$(Uf)'=\frac{d-2}{2}\frac{\mathcal{M}(y)}{y^{d-1}} \ge 0$ for $y>0$.

Moreover, 
$U(y)\to 0$ as $y\to\infty$ holds, and therefore $\|U\|_\infty<\infty$. 
To see this, we first note that the function $\mathcal{M}=\mathcal{M}(y)$
satisfies 
\eq{\label{eq;M's-estimate}
0
\le
\mathcal{M}(y)
\le \ep 2\sg_{d}y^{d-2} \quad
\text{for} \quad
y > 0}%
from relation \eqref{eq;M-sss} and the inequality $0 \le M^{*}(t,r) \le \ep 2 \sg_{d}r^{d-2}$
for $t>0$ and $r>0$.
According to inequality \eqref{eq;M's-estimate} and definition \eqref{eq;integral-factor-f},
it holds that 
$
\lim_{y \to \infty}
\frac{\mathcal{M}(y)}{\sg_{d} y^{d}} 
=0
$,
$
 f(y) \to \infty \ (y \to \infty)
$, 
$f'(y) > 0 \ (y > y_{*})$
and
$
\int_{y_{*}}^{y}
 sf(s) 
\, \ds
\to \infty \ (y \to \infty).
$
Then, we observe
$$
 \lim_{y \to \infty}
 \frac{yf(y)}{2f'(y)}
 =
 \lim_{y \to \infty}
 \frac{yf(y)}{2\big( \frac{y}{2} + \frac{\mathcal{M}(y)}{\sg_{d} y^{d-1}}\big)f(y)}
 =
 \lim_{y \to \infty}
 \frac{1}{1+ 2\frac{\mathcal{M}(y)}{\sg_{d} y^{d}}}
 =1,
$$
and thus, the de l'Hospital rule shows that
\eq{\label{eq;limit-second-term-U}
 \lim_{y \to \infty}
 \frac{\int_{\y}^{y} sf(s) \, \ds}{2f(y)}
 =
 \lim_{y \to \infty}
 \frac{yf(y)}{2f'(y)}
 =1.
}%
Combining equation \eqref{eq;U-explicit-form} 
and formula \eqref{eq;limit-second-term-U},
we obtain 
\eqn{
 \spl{
 \lim_{y \to \infty}
  U(y)
 -
 \frac{d-2}2
 =&
 \lim_{y \to \infty}
 \frac{f(\y)}{f(y)}
 \left(
  U(y_{*})
  -
  \frac{d-2}2
 \right)
 -
 \frac{d-2}{2}
 \lim_{y \to \infty}
 \frac{1}{2f(y)}
 \int_{y_{*}}^{y}
  sf(s)
 \, \ds\\
 =&
 -
 \frac{d-2}2.
 }
}
Therefore we conclude that
$\lim_{y \to \infty} U(y) = 0$, 
which yields that $U \in L^{\infty}(0,\infty)$
and
$$
 \|
  u^{*}(t)
 \|_{\infty}
 =\frac1t \| U \|_{\infty}\qquad
 \text{for}
 \quad
 t>0.
$$ 
\end{proof}

\begin{remark}
 We note that 
 the profile $\mathcal{M}$ satisfies the following upper and lower bounds:
 $$
  \underline{m}(1,y)
  \le
  \mathcal{M}(y)
  \le
  \overline{m}(1,y)
 $$
 for $y > 0$.
 Moreover, from relation \eqref{eq;m-explicit-form} and self-similarity of $m(t,r)$ in Remark \ref{rem;self-similarity-of-m},
 we obtain
 \eqn{
  \spl{
   \frac{{m}(1,y)}{y^{d-2}}
   =&
   c_{0}
   \frac{1}{\Gamma( \frac{\lambda -d +3}{2} )}
   \int_{0}^{\frac{y^{2}}{4}}
   \bigg(
    1
    -
    \frac4{y^{2}}
    z
   \bigg)^{\frac{d}2-1}
   z^{\frac{\lambda -d +3}{2}-1}
   {\rm e}^{-z}
   \, {\rm d}z
   \to
   c_{0} 
   \quad (y \to \infty)
  } }
 for $n-3 < \lambda \le n-1$.
 Thanks to this fact,
 it holds that
 $$
  \lim_{y \to \infty}
  \frac{\underline{m}(1,y)}{y^{d-2}}
  =
  \lim_{y \to \infty}
  \frac{\overline{m}(1,y)}{y^{d-2}}
  =\ep 2\sg_{d},
 $$
 and thus 
 $$
  \lim_{y \to \infty}
  \frac{\mathcal{M}(y)}{y^{d-2}}
  =\ep 2\sg_{d}.
 $$
 By applying the same argument of the last part of the proof for Lemma \ref{lem;profile-boundedness}, we observe
 $$
  \lim_{y \to \infty} 
   y^{2}U(y)
  =
  \ep 2(d-2),
 $$
 which finally shows that
 $
  u^{*}(t,\cdot )
  \to
  \ep u_{C}
 $
 uniformly on compact subsets in $\re^{d} \setminus \{ 0 \}$ as $t \searrow 0$.
 Furthermore, we observe 
 \eq{\label{eq;initial-value-of-derivative-U}
  U'(0) 
  :=
  \lim_{y \searrow 0}
  U'(y)
  = 0 
  \quad
  {\rm and}
  \quad
  U''(0)
  =
  -\frac1n
  U(0)
  \big(
   1
   +
  U(0)
  \big)
  <0,
 }
 where $U(0) := \lim_{y \searrow 0} U(y)$.
 Indeed, 
 the de l'Hospital rule and relation \eqref{u-sss} show that
 $$
  \lim_{y \searrow 0} 
  \frac{\mathcal{M}(y)}{\sg_{d}y^{d-1}}
  =
  \frac1{\sg_{d}}
  \lim_{y \searrow 0} 
  \frac{\mathcal{M}'(y)}{(d-1)y^{d-2}}
  =
  \frac{1}{d-1}
  \lim_{y \searrow 0} 
  y U(y)
  =
  0
 $$
 and
 $$
  \lim_{y \searrow 0}
  U'(y)
  =
  -
  \lim_{y \searrow 0}
  \left(\frac{y}{2}+\frac{\M(y)}{\s_dy^{d-1}}\right)U(y)
  +
  \frac{d-2}{2}
  \lim_{y \searrow 0}
  \frac{\M(y)}{\s_dy^{d-1}}
  =0.
 $$
 Combining this fact with relation \eqref{eq;limit-relation-U}, we have
 \eqn{
  \spl{
   \frac{U'(y)}{y}
   =&
   -
   \bigg(
    \frac{1}{2}+\frac{\M(y)}{\sg_d y^{d}}
   \bigg)
   U(y)
   +
   \frac{d-2}{2}\frac{\M(y)}{\sg_d y^{d}}
   \to
   -
   \bigg(
    \frac12 + \frac1d U(0)
   \bigg)
   U(0)
   +
   \frac{d-2}2
   \cdot 
   \frac1d
   U(0)
  }
 }
 as $y \searrow 0$.
 Therefore we obtain the second part of desired relation \eqref{eq;initial-value-of-derivative-U}.
\end{remark}

\begin{remark}
Note that, since the o.d.e. \eqref{M-sss} for self-similar profiles --- simpler than the p.d.e. in problem \eqref{eq;local-mass} --- also has singular coefficients, this makes its direct analysis delicate as well.  
\end{remark}

\bigskip

We state a version of the Strauss type Radial Lemma (see, e.g., \cite[Radial Lemma 1]{Strauss}) to deal with radially symmetric solutions to 
the Poisson equation.
We skip its proof because it can be found in \cite{BKZ-Nsymm,BKZ-AiM}.

\begin{lemma}[{\cite[Lemma 2.1]{BKZ-Nsymm}, \cite[Lemma 4.2]{BKZ-AiM}}]\label{lem;radial-estimate-psi}
 Let $f \in L^{1}_{{\rm loc}}(\re^{d})$ be a radially symmetric function,
 such that $\psi = E_{d}*f$, solves the Poisson equation $-\Delta \psi = f$.
 Then it holds that for $x \in \re^{d}$,
 \eq{\label{eq;radial-estimate-psi}
  x \cdot \N \psi (x)
  =
  -
  \frac{1}{\sg_{d}|x|^{d-2}}
  \int_{\{ |y| \le |x| \}}
   f(y)
  \,
  \dy.
 }
\end{lemma}

Finally, we prove our main result.

\begin{proof}[Proof of Theorem \ref{main}]  
Since the function $M^{*}=M^{*}(t,r)$ is a classical solution of problem \eqref{eq;local-mass},
we may differentiate the first equation of problem \eqref{eq;local-mass} with respect to the variable $r>0$.
Then we have
\eq{\label{eq;differentiated-local-mass}
 M^{*}_{rt}-M^{*}_{rrr}+(d-1)\frac{rM^{*}_{rr}-M^{*}_{r}}{r^{2}}
 -
 \left[
     M^{*}_{r}
  \left(
  \frac{1}{\sg_{d}r^{d-1}}M^{*}_{r}
  \right)
  +
  M
 \left(
  \frac{1}{\sg_{d}r^{d-1}}M^{*}_{r}
  \right)_{r}
 \right]
 =0.
}
We note that
\eq{\label{eq;formulas-u}
 \spl{
  M^{*}_{r}
  =& 
  \sg_{d}r^{d-1}u^{*},
  \quad
  u^{*}
  =
  \frac{1}{\sg_{d}r^{d-1}}M^{*}_{r},
  \\
  M^{*}_{rt}
  =&
  \sg_{d}r^{d-1}u^{*}_{t},\quad
  M^{*}_{rr}
  =
  \sg_{d}
  \big[
   (d-1)r^{d-2}u^{*}+r^{d-1}u^{*}_{r}
  \big],\\
  M^{*}_{rrr}
  =&
  \sg_{d}
  \big[
   (d-1)(d-2)r^{d-3}u^{*}
   +
   2(d-1)r^{d-2}u^{*}_{r}
   +
   r^{d-1}u^{*}_{rr}
  \big].
 }
}
When we substitute formulas \eqref{eq;formulas-u} in equation \eqref{eq;differentiated-local-mass}, we obtain
\eqn{
 \spl{
  &\sg_{d}r^{d-1}u^{*}_{t}
  -
  \sg_{d}
  r^{d-1}
  \left[
   (d-1)(d-2)
   \frac{u^{*}}{r^{2}}
   +
   2(d-1)
   \frac{u^{*}_{r}}{r}
   +
   u^{*}_{rr}
  \right]\\
  &
  +
  (d-1)
  \sg_{d}
  r^{d-1}
  \frac{\big[(d-1)u^{*}+ru^{*}_{r}\big]-u^{*}}{r^{2}}
  -
  \sg_{d}
  r^{d-1}
  \left[
   (u^{*})^{2}
   +
   u^{*}_{r}
   \frac{1}{r^{d-1}}\int_{0}^{r}u^{*}(t,s) s^{d-1} \, \ds
  \right]
  =0
 }
}
or equivalently
\eq{\label{eq;u-radial}
 \spl{
 u^{*}_{t}
 -
 u^{*}_{rr}
 -
 \frac{d-1}{r}
 u^{*}_{r}
 -
  (u^{*})^{2}
 -
  u^{*}_{r}
  \frac{1}{r^{d-1}}
  \int_{0}^{r} 
   u^{*}(t,s) s^{d-1}
  \, \ds
  =0.
  }
}
Applying formula \eqref{eq;radial-estimate-psi} of Lemma \ref{lem;radial-estimate-psi} to radially symmetric solutions $\psi^{*}=\psi^{*}(t,x)$, we have
\eq{\label{eq;Strauss-est}
 \spl{
 \psi^{*}_{r}(t,r)
 =
 \frac{x}{|x|} \cdot \N_{x} \psi^{*}(t,x)
 =&
 -
 \frac{1}{\sg_{d}|x|^{d-1}}
 \int_{\{ |y| \le |x|\}}
  u^{*}(t,y)
 \,
 \dy\\
 =&
 -
 \frac1{r^{d-1}}
 \int_{0}^{r}
 u^{*}(t,s) s^{d-1}
 \, \ds,
 }
}
where $r=|x|$.
Combining formula \eqref{eq;Strauss-est} with equation \eqref{eq;u-radial},
we observe
$$
 -u^{*}_{r}(t,r)
  \frac{1}{r^{d-1}}
  \int_{0}^{r} 
   u^{*}(t,s) s^{d-1}
  \, ds
  =
  u^{*}_{r}(t,r)\psi^{*}_{r}(t,r)
  =
  \N_{x} u^{*}(t,x) \cdot \N_{x} \psi^{*}(t,x)
  ,
$$
and thus
\eqn{
 \spl{
  u^{*}_{t}-\Delta u^{*}+\N \cdot (u^{*}\N \psi^{*})
  =& u^{*}_{t}-\Delta u^{*}+u^{*}\Delta \psi^{*} + \N u^{*} \cdot \N \psi^{*}\\
  =&  u^{*}_{t}-u^{*}_{rr}-\frac{d-1}{r}u^{*}_{r}-(u^{*})^{2}+u^{*}_{r}\psi^{*}_{r}=0,
 }
}
which indicates the first equation of problem \eqref{eq;DD}. 
Estimate \eqref{U:R} is a direct consequence of estimate \eqref{eq;apriori-M^{K}}.
Moreover, the self-similar property of solution $u^{*}=u^{*}(t,x)$ is implied by that of the function $M^{*}=M^{*}(t,r)$, see also \rf{eq;M-sss-2}.
The profile $U=U(x)$ constructed in Lemma \ref{lem;profile-boundedness} 
is in the class $\mathcal{C}^{\infty}(\re^{n}) \cap L^{\infty}(\re^{n})$ by Lemma \ref{lem;profile-boundedness} and the standard application of the parabolic  regularity argument (cf. Theorem \ref{glo}). 
Hence, we complete our proof of Theorem \ref{main}. 
\end{proof}


\section{Nonexistence of self-similar solutions}\label{sec:non}

In this last section, we give the proof of relation \eqref{bound-C(d)}.

\begin{proposition}\label{prop;bounds-C(d)}
System \rf{eq;DD} supplemented with the initial condition \rf{u0:m} 
cannot have any local-in-time solution if 
$\varepsilon > C(d)$. 
Moreover, for $d \ge 3$, it holds that 
\eqn{
1<\frac{2}{d-1}\left(\frac{\Gamma\left(\frac{d+1}{2}\right)}{\Gamma\left(\frac{d}{2}\right)}\right)^2 < C(d) <\left(\frac{2}{d-2}\right)^{\frac12}\frac{\Gamma\left(\frac{d+1}{2}\right)}{\Gamma\left(\frac{d}{2}\right)} <\frac{d-1}{d-2}\le 2. 
}

\end{proposition}
\begin{proof} 
It follows from the blowup criterion for equation \rf{eq;DD} in \cite[Theorem 2.2]{BZ-JEE} that if $\varepsilon > C(d)$, then solutions of equation \rf{eq;DD} with $u_0=\varepsilon u_C$ cannot exist for any $T>0$. 
Indeed, according to this criterion if 
$$
T{\rm e}^{T\Delta}u_0(0)>C(d)
$$
for some $T>0$, then the solution with $u_0$ as the initial datum blows up before time $T$. For $u_0=\ep u_C$ this leads to $\ep> C(d)$. 

Clearly, for $d\ge 3$ we have $C(d)\in(1,2)$ as was proved in \cite{BZ-JEE} but a more precise, yet simple, estimate \rf{bound-C(d)} for $C(d)$ is available. 
To prove the upper bounds, observe that by the inequalities between harmonic, geometric and arithmetic means, the denominator of the integrand in \rf{C(d):0} satisfies 
$$
\frac{1}{2(d-2)+4\r^2}\le \frac{1}{4\sqrt{2(d-2)}\r}\le \frac14\left(\frac{1}{2(d-2)}+\frac{1}{4\r^2}\right)
$$
with  a strict inequality whenever $2(d-2)\ne 4\r^2$. 
Then, we have 
$$
C(d)< \frac{16}{\Gamma\left(\frac{d}{2}\right)}\frac{1}{4\sqrt{2(d-2)}}\int_0^\infty {\rm e}^{-\r^2}\r^d \drho = \left(\frac{2}{d-2}\right)^{\frac12} \frac{\Gamma\left(\frac{d+1}{2}\right)}{\Gamma\left(\frac{d}{2}\right)}.
$$
The lower bound for $C(d)$ is obtained using a consequence of the Cauchy inequality
\bea
\left(\int_0^\infty{\rm e}^{-\r^2}\r^d\drho\right)^2 &\le& \int_0^\infty{\rm e}^{-\r^2}\frac{\r^{d+1}}{2(d-2)+4\r^2}\drho\times \int_0^\infty{\rm e}^{-\r^2}\r^{d-1}\left(2(d-2)+4\r^2\right)\drho\nonumber  \\
&=&\frac{\Gamma\left(\frac{d}{2}\right)}{16}C(d)\frac12\left(2(d-2)\Gamma\left(\frac{d}{2}\right)+4\frac{d}{2}\Gamma\left(\frac{d}{2}\right)\right)\nonumber
\eea 
so that 
$$
C(d)\ge \frac14\Gamma\left(\frac{d+1}{2}\right)\frac{16}{\Gamma\left(\frac{d}{2}\right)} \frac{2}{2(d-2)\Gamma \left(\frac{d}{2}\right)+4\frac{d}{2}\Gamma\left(\frac{d}{2}\right)} =\frac{2}{d-1} \left(\frac{\Gamma\left(\frac{d+1}{2}\right)}{\Gamma\left(\frac{d}{2}\right)}\right)^2 . 
$$
\end{proof}

\section*{Acknowledgments}
The authors wish to express their gratitude to Adam Nowak 
 for helpful comments concerning Bessel semigroups. 
The  first and the third named authors were supported by the Polish NCN grant \hbox{2016/23/B/ST1/00434}.


\end{document}